\documentclass[a4paper,twoside,11pt]{article}
\usepackage[english]{babel}
\usepackage[latin1]{inputenc}
\usepackage{indentfirst}
\usepackage{amssymb}
\usepackage{amsmath}
\usepackage{amsthm}

\Large
\theoremstyle{plain}
\newtheoremstyle{miestilo}{12pt}{\topsep}{\itshape}{}{\bf}{}{ }{}
\swapnumbers
\theoremstyle{miestilo}
\newtheorem{teorema}[subsection]{Theorem.}
\newtheorem{proposicion}[subsection]{Proposition.}
\newtheorem{lema}[subsection]{Lemma.}
\newtheorem{corolario}[subsection]{Corollary.}
\newtheoremstyle{misnotas}{12pt}{12pt}{}{}{\bf}{}{ }{\remark}
\theoremstyle{misnotas}
\newtheorem{remark}[subsection]{\ {\bf Remark.}}

\newtheorem{apartado}[subsection]{\ {\ }}
\allowdisplaybreaks[2]
\begin{document}
\flushbottom
%
%

  {\LARGE\centerline{\bf The ideal of Lesieur-Croisot elements of Jordan pairs}}
 \medskip

\centerline{ Fernando Montaner \footnote{Partially supported   by grant MTM2017-83506-C2-1-P (AEI/FEDER, UE), and by
Diputaci\'on General de Arag\'on (Grupo de Investigaci\'{o}n
\'{A}lgebra y Geometr\'{\i}a).}} \centerline{{\sl Departamento de Matem\'{a}ticas,
Universidad de Zaragoza}} \centerline{{\sl 50009 Zaragoza,
Spain}} \centerline{E-mail: fmontane@unizar.es} \centerline{and}
\centerline{ Irene Paniello \footnote{Partially supported   by grant MTM2017-83506-C2-1-P (AEI/FEDER, UE).}} \centerline{{\sl
Departamento de Estad\'{\i}stica, Inform\'{a}tica y  Matem\'{a}ticas, Universidad
P\'{u}blica de Navarra}}
 \centerline{{\sl 31006 Pamplona, Spain}}
\centerline{E-mail: irene.paniello@unavarra.es}

 \medskip

\begin{abstract}
We study the sets of elements of Jordan pairs whose local Jordan algebras are Lesieur-Croisot algebras, that is, classical orders in nondegenerate Jordan algebras with finite capacity. It is then proved that, if the Jordan pair is nondegenerate, the set of its Lesieur-Croisot elements is an ideal of the Jordan pair.
\end{abstract}


\section*{Introduction}

Two obstacles to be overcome in order to embed a semiprime associative algebra into a semisimple artinian algebra are singularity and uniform dimension. Once we are able to ensure both nonsingularity and finite uniform dimension for a semiprime associative algebra, Goldie theory gives us the
existence of its classical algebra of quotients. Nevertheless  there is a number of conditions, equivalent to being  nonsingular and having finite uniform dimension,  that also imply that any semiprime associative algebra satisfying any of these equivalent conditions is an order in a  semisimple artinian algebra. One of these equivalent conditions, more precisely, that given in terms of the essentiality of one-sided ideals of semiprime Goldie algebras as   exactly those containing regular elements,  becomes particularly outstanding in the extension of Goldie's Theorems to Jordan algebras.

First Zelmanov \cite{z-gol1,z-gol2} for linear Jordan algebras and later Fernández López, García Rus and Montaner \cite{fgm} for quadratic Jordan algebras considered the extension of Goldie theory to Jordan algebras. In \cite{fgm}, besides characterizing  nodegenerate Jordan Goldie algebras as those having a nondegenerate artininan classical algebra of quotients,  the authors translated to the Jordan algebras the above mentioned equivalent characterizations of semiprime associative Goldie algebras. However it turned out that the previously mentioned condition  involving  essential one-sided ideals did not hold for Jordan algebras. Indeed, nondegenerate Jordan algebras whose essential inner ideals can be  characterized as those containing injective elements are not Goldie, since they cannot be always imbedded into nondegenerate artinian algebras, but into nondegenerate Jordan algebras having finite capacity. These algebras were named Lesieur-Croisot Jordan algebras by Montaner and Tocón in \cite{mt-lc}.

In \cite{mt-lc} the authors studied not only Jordan algebras having the property of being  Lesieur-Croisot, but also the set of elements of a nondegenerate Jordan algebra whose local algebras are  Lesieur-Croisot, thus introducing what was called in that reference a local Lesieur-Croisot theory of Jordan algebras.

Recall that since they were introduced by Meyberg in \cite{meyberg}, local algebras have been related to some of the most important problems arising  in Jordan theory since, as a result of a series of local-to-global inheritance procedures, local properties of   local algebras can be carried back to the original Jordan algebra, pair or triple system. For instance, the study of Jordan systems with nonzero local PI-algebras have been a recurrent topic appearing in Jordan theory since Zelmanov's work on pairs and triple systems.

In the associative pair setting Fernández López, García Rus, Gómez Lozano, and Siles Molina established in \cite{goldie pares} an analogue of  the classical Goldie theory for semiprime associative pairs, where orders are defined locally.  A similar local approach was followed by  Fernández López, García Rus, and Jaa in \cite{towards}, where the authors  introduced a notion of order in nondegenerate linear Jordan pairs  with descending chain condition on principal inner ideals. Later Fernández López, García Rus, and Montaner provided a first approach to a Goldie theory for quadratic Jordan pairs in their unpublished manuscript \cite{fgm_pares}.

Following this local approach, and taking into account \cite{mt-lc},   in this paper we study the set of elements of a Jordan pair which give rise to local algebras that are classical orders in nondegenerate Jordan algebras of finite capacity, i.e., local algebras which are Lesieur-Croisot   algebras. The main result we prove in the present paper is that this set is an ideal when the Jordan pair is nongenerate.

The paper is organized as follows. After this introduction, we devote the first section to basic notions and preliminaries. Then, in Section 2, we include   a series of results on the  set of elements of a  Jordan pair having essential annihilator. This set was proved to be an ideal of any nondegenerate Jordan pair by Fernández López, García Rus and Montaner   \cite{fgm_pares}. Besides formulating the above mentioned result, we also include some properties relating the singular ideal of a nondegenerate Jordan pair  to that  of any of its local algebras.

Section 3 is devoted to  Lesieur Croisot Jordan algebras. These Jordan algebras were characterized by Fernández López, García Rus and Montaner  in \cite{fgm}. We also recall here the notion of Lesieur Croisot element as introduced by Montaner and Toc\'{o}n in \cite{mt-lc}.

Sections 4 and 5 contain some results that will be instrumental later. More precisely, Section 4  focuses on associative pairs and their left and right singular ideals. Recall that associative pairs naturally arise as associative envelopes of special Jordan pairs, and therefore they are of great relevance in the study strongly prime Jordan pairs  of hermitian type. We revisit   some results introduced in \cite{goldie pares} and \cite{gs comm alg32}  focusing on the relationship between the left or right singular ideal of a semiprime associative pair and the ones of its local algebras. We also consider the relationship between the sets of elements of a semiprime associative pair having finite left or right  Goldie (or uniform) dimension and the annihilator of its left and right singular ideals.
Then, as a first approach to  study the set of Lesieur-Croisot elements of nondegenerate Jordan pairs we consider in Section 5 the set of PI-elements of strongly prime Jordan pairs of hermitian type. More precisely, we consider the relationship between the set of PI-elements of a strongly prime Jordan pair   of hermitian type and those of any of its $\ast$-tight associative pair envelopes, thus extending to Jordan pairs a result previously given by Montaner in \cite{pi-i} for Jordan algebras.

Finally, in the sixth and last section,  we address the prove that the set of LC-elements of a nondegenerate Jordan pair is an ideal.
Here we follow the pattern given in \cite{mt-lc}, and after   dealing with   strongly prime Jordan pairs having non-zero PI-elements,  we study the  strongly prime PI-less case. Though the  steps to follow are rather similar to those in \cite{mt-lc},   there are still some intermediate results on Jordan pairs that need to be proved since in \cite{mt-lc} the authors made use of a series of  facts  proved \cite{fgm} only for Jordan algebras. Finally, the extension from strongly prime to nondegenerate Jordan pairs yields in the decomposition of local algebras of  nondegenerate Jordan pairs at nonzero  LC-elements into essential subdirect sums of strongly prime Jordan algebras, as a consequence of the fact that  LC-elements of a nondegenerate Jordan pair are, in particular, semi-uniform elements.

%
%

\section{Preliminaries}

\begin{apartado} We will work with  associative and Jordan   systems (algebras and pairs) over a unital commutative ring of scalars $\Phi$ which will be fixed throughout. We refer to \cite{loos-jp,meyberg,mcz} for notation, terminology and basic results. For us, Jordan system will mean a quadratic Jordan system. We will use identities for Jordan pairs proved in \cite{loos-jp}, which will be quoted with the labellings JPn. \end{apartado}

\begin{apartado} A {\it Jordan algebra} $J$ has products $U_xy$ and $x^2$, quadratic in $x$ and linear in $y$, whose linearizations are  $U_{x,z}y=V_{x,y}z=\{x,y,z\}=U_{x+z}y-U_xy-U_zy$ and $x\circ y=V_xy=(x+y)^2-x^2-y^2$. A {\it Jordan pair} $V=(V^+ ,V^{-  })$ has products $Q_xy$ for
$x\in V^\sigma $, $y\in V^{-\sigma }$, \, $\sigma =\pm$, with linearizations $Q_{x,z}y=D_{x,y}z=\{x,y,z\}=Q_{x+z}y-Q_xy-Q_zy$.  This  introductory section mainly focuses on Jordan pairs. The reader is referred to \cite[Section 0]{mp-localorders} for the corresponding basic preliminary results on Jordan algebras. \end{apartado}

\begin{apartado} An {\sl associative pair} is a pair $A=(A^+,A^-)$ of $\Phi$-modules  together with trilinear maps $  \  A^\sigma \times A^{-\sigma}  \times A^\sigma \ \to \ A^\sigma$, $\sigma=\pm$, defined by $(x,y,z)\  \mapsto \  xyz$ and such that  $ (xyz) uv =x(yzu)v =xy(zuv) $  for all $x,z,v\in A^\sigma$, $y,u\in A^{-\sigma}$, $\sigma=\pm$.  An {\sl (polarized) involution} $\ast$ in an associative pair $A$ consists of two linear maps $\ast : A^\sigma \to A^\sigma$ such that $(x^\ast)^\ast=x$ and $(xyz)^\ast=z^\ast y^\ast x^\ast$ for all $x,z\in A^\sigma$, $y \in A^{-\sigma}$, $\sigma=\pm$.\end{apartado}

\begin{apartado} Similarly to the algebra case, where every associative algebra $A$ gives rise to a Jordan algebra $A^{(+)}$ by taking $U_xy=xyx$ and $x^2=xx$ for $x,y\in A$, any associative pair $A=(A^+,A^-)$ provides a Jordan pair $A^{(+)}$ taking $Q_xy=xyx$ for all $x\in A^\sigma$, $y \in A^{-\sigma}$, $\sigma=\pm$. A Jordan system (algebra or pair) is {\sl  special} if it is isomorphic to a Jordan subsystem of $A^{(+)}$  for some associative system  $A$ and {\sl $i$-special } if it satisfies all identities satisfied by all special Jordan systems.\end{apartado}

\begin{apartado} If $A$ is an associative algebra with involution $\ast$, then $H(A,\ast)=\{a\in A\mid a^\ast=a\}$ is a Jordan subalgebra of $A^{(+)}$ and so are {\sl ample subspaces} $H_0(A,\ast)$ of symmetric elements of $A$ i.e. subspaces of $H(A,\ast)$ such that $\{a\}=a+a^\ast$, $aa^\ast$ and $aha^\ast$ are in $H_0(A,\ast)$ for all $a\in A$ and all $h\in  H_0 (A,\ast)$. If $A=(A^+,A^-)$ is an associative pair with polarized involution $\ast$, then $H(A,\ast)=\big(H(A^+,\ast),H(A^-,\ast)\big) $, where $H(A^\sigma,\ast)=\{a\in A^\sigma\mid a^\ast=a\}$ is a Jordan subpair of $A^{(+)}$. {\sl  Ample subspairs} $H_0(A,\ast)\subseteq H(A,\ast)$  of symmetric elements containing the traces  $\{a\}=a+a^\ast$ of elements of $A^\sigma$ and satisfying $a H_0(A^{-\sigma},\ast) a^\ast\subseteq H_0(A,\ast) $ for all $a\in A^\sigma$, $\sigma=\pm$ are also Jordan subpairs of $A^{(+)}$.
\end{apartado}

\begin{apartado}  A $\Phi$-submodule $K$ of $V^\sigma $ is an {\it inner ideal} of $V$ if $Q_kV^{-\sigma } \subseteq K$ for all $k\in K$.
 A $\Phi$-submodule $I=(I^+,I^-)$ is an {\it ideal} if $Q_{I^\sigma }V^{-\sigma }+Q_{V^\sigma }I^{-\sigma }+\{I^\sigma , V^{-\sigma }, V^\sigma \} \subseteq I^\sigma $ for $\sigma =\pm$.   An  ideal   of $V$ is  {\it essential}  if it intersects nontrivially any nonzero ideal of $V$. Essentiality of an inner  ideal  $ K$ of $V^\sigma $ implies that it hits nontrivially any other nonzero inner ideal of $V$, contained in $V^\sigma $.
 \end{apartado}

\begin{apartado}  A Jordan pair $V$ is {\sl nondegenerate} if $Q_xV^\sigma\neq0$ for any nonzero $x\in V^{-\sigma}$, $\sigma=\pm$, and {\sl prime} if $Q_IL\neq0$ for any nonzero ideals $I$ and $L$ of $V$. A Jordan pair is {\sl strongly prime } if $V$ is both nondegenerate and prime. A Jordan pair is {\sl semiprime} if $V$ has no  nonzero trivial ideals i.e. if $Q_II\neq 0$ for any nonzero ideal $I$   of $V$.
 \end{apartado}

\begin{apartado}\label{derived ideals}  The product $Q_IL$ of two ideals $I$ and $L$ of a Jordan pair $V$ is not an ideal of $V$, but a   semi-ideal. Recall that  a $\Phi$-submodule $I=(I^+,I^-)$  of $V$ is a   {\it semi-ideal}  if $Q_IV+D_{V,V}I+Q_VQ_VI\subseteq I$. The product $Q_IL$ of semi-ideals is again a semi-ideal    \cite{mc-prime-inh}.
For any semi-ideal $I$  we can form   a chain of semi-ideals (called {\sl intrinsic derived spaces}) given by  $ I^{\langle0\rangle}=I$, $I^{\langle1\rangle}=Q_II $
 and $I^{\langle n\rangle}=Q_{I^{\langle n-1\rangle}}I^{\langle n-1\rangle}$, for all $n\geq 1$.
On the other hand,  if $I$ and $ L$ are ideals of $V$, then so is   $I\ast L=Q_IL+ Q_VQ_IL$ and, in particular,   $I^{(1)}=I\ast I$ the {\it derived ideal } of $I$ and the {\sl nth derived ideals} $I^{(n)}=I^{(n-1)}\ast I^{(n-1)}$   for all $n\geq 1$.

We remark  that given a semi-ideal $I$ of $V$, then $I+Q_VI$ is an ideal of $V$ called the {\it hull of $I$} and denoted $I^{(1\rangle}$. Similarly  we can consider $I^{(n\rangle}=I^{\langle n \rangle}+Q_VI^{\langle n \rangle}$ which are ideals of $V$ for any semi-ideal $I$.
  For any ideal $I$ of $V$, it holds that $I^{  ( n+1 \rangle}\subseteq I^{\langle n \rangle} \subseteq I^{( n \rangle} \subseteq I^{( n )}   $    for any $n\geq 0$  \cite[p.~210-211]{da-hermitian}. Moreover the following conditions are equivalent for any ideal $I$ of $V$ and any $n\geq 1$:
  \begin{enumerate}
  \item[(i)] $V$ is semiprime that is $I^{\langle 1\rangle} = 0$ implies $I=0$ for any ideal $I$ of $I$.
  \item[(ii)] $I^{(n)}=0 \Rightarrow I=0$ for any ideal $I$ of $I$.
 \item[(iii)] $I^{\langle n \rangle}=0 \Rightarrow I=0$  for any ideal $I$ of $I$.
  \item[(iv)] $I^{( n \rangle}=0 \Rightarrow I=0$ for any ideal $I$ of $I$.
  \end{enumerate}
 \end{apartado}

\begin{apartado}\label{annihilators}
The {\it annihilator} of a  subset $X\subseteq V^\sigma  $ is the set $ann_V(X)$ of all $a\in V^{-\sigma }$ such that $Q_ax=Q_xa=0$,
$Q_aQ_xV^{-\sigma }=D_{a,x}V^{-\sigma }=0$ and $Q_xQ_aV^\sigma=D_{x,a}V^\sigma =0$ for all $x\in X$.   Annihilators are always inner
ideals.  If $I=(I^+,I^-)$ is an ideal of $V$, then $ann_V(I)=(ann_V(I^-),ann_V(I^+))$ is again an ideal of $V$. If $V$ is nondegenerate, the
annihilator of an ideal $I$ of $V$ is characterized as follows (see \cite[Proposition 1.7(i)]{mc-prime-inh} and \cite[Lemma 1.3]{pi-ii}):
$$ann_V(I^\sigma )=\{a\in V^{-\sigma }\mid Q_aI^\sigma =0\}=\{a\in V^{-\sigma }\mid Q_{I^\sigma}a=0\}.$$ An ideal of a nondegenerate Jordan pair is essential if and only if it has zero annihilator.
\end{apartado}

\begin{apartado}  Let $V$ be a  special Jordan pair. An associative pair with a (polarized) involution $(A,\ast)$ is an {\it associative $\ast$-envelope} of $V$  if $A$ is generated by $V\subseteq H(A, \ast)$, and it is said to be {\it $\ast$-tight over $V$} if any nonzero $\ast$-ideal $I$ of $A$ hits $V$ nontrivially: $I\cap V\neq0$ for any nonzero $\ast$-ideal $I$ of $A$.
 \end{apartado}

\begin{apartado}  An $i$-special Jordan pair $V$ is {\it of hermitian type} if its hermitian parts are not annihilated, i.e., $ann_V(\sum_{\cal H}{\cal H}(V))=\bigcap_{\cal H}\  ann_V({\cal H}(V))=0$. Clearly  a prime Jordan pair is of hermitian type if and only if ${\cal H}(V)\neq0$
for some hermitian ideal ${\cal H}(X)$ of the free special Jordan pair $FSJP[X]$.
 \end{apartado}

\begin{apartado}  Given a Jordan pair $V=(V^+,V^-)$ and an element $a\in V^{-\sigma}$ the $\Phi$-module $V^\sigma$ becomes a Jordan algebra denoted by $(V^\sigma)^{(a)}$ and called the $a$-{\it homotope} of $V$ defining $U_x^{(a)}y=Q_xQ_ay$ and $x^{(2,a)}=Q_xa$ for all $x,y\in V^\sigma$. The set $Ker_V(a)=Ker \, a=\{x\in V^\sigma\mid Q_ax=Q_aQ_xa=0\}$ is an ideal of  $(V^\sigma)^{(a)}$ and the quotient $V_a^\sigma=  (V^\sigma)^{(a)}/Ker\,a$ is a Jordan algebra called the {\it local algebra  of $V$ at $a$}. If $V$ is nondegenerate, then $ Ker \, a=\{x\in V^\sigma\mid Q_ax= 0\}$.
\end{apartado}

\begin{apartado}\label{socle}   The socle $Soc(V)$ of a Jordan pair was studied by Loos in \cite{loos-socle}. The socle of a nondegenerate Jordan pair $V$ is the sum of all minimal inner ideals of $V$, it is a direct sum of simple ideals of $V$, consists of von Neumann regular elements, and satisfies the dcc on principal inner ideals \cite{loos-finiteness}. The elements of the socle of a nondegenerate Jordan pair are exactly those whose local algebras have finite capacity \cite[Lemma 0.7(b)]{pi-i}. A characterization of the elements of the socle of a nondegenerate Jordan pair in terms of the descending chain condition of principal ideals can be found in \cite[Theorem 1]{loos-socle}.
\end{apartado}

\begin{apartado}
An element $a$ of a Jordan pair  $V$ is a PI-{\it element} if the local  of $V$  at $a$ is a PI-algebra. If $V$ is a Jordan pair, we write $PI(V)=(PI(V^+),PI(V^-))$ where $PI(V^\sigma)$ is the set of PI-elements of $V^\sigma$, $\sigma=\pm$. It was proved in  \cite[Theorem 5.4]{pi-i} that the set of  PI-elements of a nondegenerate Jordan pair   is an ideal. Jordan pair with $PI(V)=0$ are called {\sl PI-less} Jordan pairs. PI-less strongly prime Jordan pairs are special of hermitian type.\end{apartado}

\section{The singular ideal of a Jordan pair}

In this section we recall the definition of the singular set for Jordan pairs, which is analogous to that given in \cite{fgm} for Jordan algebras. The contents of this section correspond to an unpublished work by  the first author together with Fernández López and García Rus \cite{fgm_pares}. The main result we provide here is that the set of elements of a Jordan pair having essential annihilator form an ideal when the Jordan pair is nondegenerate.

\begin{apartado}   The {\it singular set} of a Jordan pair $V$ is $\Theta(V)=( \Theta(V^+), \Theta(V^-))$, where $$ \Theta(V^\sigma) =\{x\in V^\sigma\mid ann_V(x)\ \hbox{is essential}\}.$$ A Jordan pair $V$ will be called {\it nonsingular} if $\Theta(V)=0$. If $\Theta(V)=V$, then the Jordan pair $V$ is   called {\it singular}.

We note  that since for all $x,y\in V^\sigma$, $ann_V(x)\cap ann_V(y)\subseteq ann_V(x+y)$ and $ann_V(x)\subseteq ann_V(Q_xV^{-\sigma})$, $\Theta(V)$ consists of a pair of inner ideals of $V$.
 \end{apartado}

\begin{proposicion}\label{JP ideal singular prop} {\rm \cite[Proposition 3.1, Propostion 3.2]{fgm_pares}}  Let $V$ be a nondegenerate Jordan pair and $b\in V^{-\sigma}$.
\begin{enumerate}
\item[(i)] $ (\Theta(V^\sigma)+Ker\, b)/Ker\, b$ is an inner ideal of $V_b^\sigma$ contained in
  $  \Theta(V^\sigma_b)$.
\item[(ii)]  If $\overline{x}\in \Theta(V_b^\sigma)$, then
  $Q_bx\in \Theta(V^{-\sigma})$.
\end{enumerate}
\end{proposicion}

\begin{teorema}\label{ideal singular} {\rm \cite[Theorem 3.3]{fgm_pares}} Let $V$ be a nondegenerate Jordan pair. Then $\Theta(V)$ is an ideal of $V$, called the singular ideal of $V$.
\end{teorema}

Next we bring together, for further use and citation, different results involving the singular ideal $\Theta(V)$ of a nondegenerate Jordan pair $V$. Even though proofs are not included (but given in the unpublished reference \cite{fgm_pares}) all them closely follow  to those of the analogous results for nondegenerate Jordan algebras, which can be found   in \cite{fgm}.

\begin{corolario}\label{corolario ideal singular} Let $V$ be a nondegenerate Jordan pair. Then:
 \begin{enumerate}
  \item[(i)] Every $b\in V^{-\sigma}$ such that $V_b^\sigma$ is nonsingular is contained in $ann_V(\Theta(V^\sigma))$.
 \item[(ii)] $V$ is nonsingular if and only if for every   $b\in V^{-\sigma}$, the local algebra $V^\sigma_b$ of $V$ at $b$ is nonsingular.
\item[(iii)]  If $V$ is strongly prime and $V_b^\sigma$ is nonsingular for a nonzero $b\in V^{-\sigma}$, then $V$ is nonsingular.
 \item[(iv)] $\Theta(V)\cap PI(V)=0$. In particular, if $V$ is strongly prime and $PI(V)\neq0$ then $V$ is nonsingular.
     \item[(v)] For any ideal $I$   of $V$,   $\Theta(V)\cap I=\Theta(I)$. Thus, if $I$ is essential,
     $V$ is nonsingular if and only if $I$ is nonsingular.
         \item[(vi)]  $\Theta(V)$ does not contain nonzero von Neumann regular elements. Thus, any nondegenerate Jordan pair with essential socle is nonsingular.
  \end{enumerate}
\end{corolario}
\begin{proof} (i) corresponds to \cite[Remark 3.5]{fgm_pares}, (ii) and (iii) are \cite[Corollary 3.6(i)-(iii)]{fgm_pares}, then \cite[Corollary 3.7]{fgm_pares} and \cite[Proposition 3.8]{fgm_pares} give (iv) and (v).  Finally (vi) follows from \cite[Proposition 3.10]{fgm_pares}
 and \cite[Corollary 3.11]{fgm_pares}.
\end{proof}

\section{Lesieur-Croisot Jordan algebras}

This section contains some of the main basic definitions and results relative to Goldie theory for quadratic Jordan algebras as well as the notion of Lesieur Croisot Jordan algebra. We refer the reader to \cite{fgm} and \cite{mt-lc,lc-pr} for more accurate results on this theory.

\begin{apartado}\label{denominadores-JP} Let $\widetilde{J}$ be a Jordan algebra,  $J$  be  a subalgebra of $\widetilde{J}$, and   $\widetilde{a}\in\widetilde{J}$. We recall from \cite{pi-ii} that an element $x\in J$ is a \emph{$J$-denominator} of
$\widetilde{a}$ if the following multiplications take $\widetilde{a}$ back into $J$:

 \centerline{ \begin{tabular}{lll}
   (Di)\ $U_x\widetilde{a}$ &  (Dii)\ $U_{\widetilde{a}}x$ &  %
   (Diii)\ $U_{\widetilde{a}}U_x \widehat{J}$ \\
   (Diii')\ $U_xU_{\widetilde{a}}\widehat{J}$  & %
   (Div)\ $V_{x,\widetilde{a}} \widehat{J}$  & (Div')\ $V_{ \widetilde{a},x}\widehat{J}$, \\
 \end{tabular}}
 \noindent  where $\widehat{J}$ denotes the unital hull $\widehat{J}=\Phi1\oplus J$ of $J$. We will denote the set of $J$-denominators of $\widetilde{a}$ by ${\cal D}_J(\widetilde{a})$. It has been proved in \cite{pi-ii} that ${\cal D}_J(\widetilde{a})$ is an inner ideal of $J$.
\end{apartado}

\begin{apartado}\label{new_ii} Let $J$ be a Jordan algebra, $K$ be  an inner ideal of $J$ and $a\in J$. Following \cite{densos, esenciales} we consider the sets $  (K:a)_L=\{x\in K\  \mid   x\circ a\in K  \} $ and $ (K:a)=\{x\in K\  \mid U_ax,\ x\circ a\in K  \}$, both   inner ideals of $J$ for all $a\in J$.  Inductively, for any   finite family of elements $a_1,\ \ldots, a_n\in J$, we also define $(K:a_1:a_2:\ldots:a_n)=((K:a_1:\ldots
a_{n-1}): a_n)$.\end{apartado}

\begin{apartado}\label{nii} An inner ideal $K$ of $J$ is said to be \emph{dense} if $U_c(K:a_1:a_2:\ldots:a_n)\neq0$ for any finite collection of
elements  $a_1,\ \ldots, a_n\in J$ and any $0\neq c\in J$.
\end{apartado}

\begin{apartado}\label{aq} Let $J$ be a subalgebra of a Jordan algebra $Q$. Following \cite{esenciales,densos}, we say that $Q$ is a \emph{general algebra of quotients of $J$} if the following conditions hold:
\begin{enumerate}\item[(AQ1)] ${\cal D}_J(q)$ is a dense inner ideal of $J$ for all $q\in Q$. \item[(AQ2)] $U_q {\cal D}_J(q)\neq0$ for any nonzero $q\in Q$.\end{enumerate}
\end{apartado}

\begin{apartado}\label{cl-q} A nonempty subset $S\subseteq J$ is a \emph{monad} if $U_st$ and $s^2$ are in $S$ for all $s,t\in S$. A
subalgebra $J$ of a unital Jordan algebra $Q$ is an \emph{$S$-order in $Q$} or equivalently $Q$ is a \emph{$S$-algebra of quotients, or an algebra of fractions (of $J$ relative to $S$)} if:
\begin{enumerate}
\item[(ClQ1)] every element $s\in S$ is invertible in $Q$.
\item[(ClQ2)] each $q\in Q$ has a $J$-denominator in $S$.
\item[(ClQ3)] for all $s,t\in S$, $U_sS\cap U_tS\neq \emptyset$.
\end{enumerate}
\end{apartado}

\begin{apartado} An element $s$ of a Jordan algebra $J$ is said to be \emph{injective} if the mapping $U_s$ is injective over $J$. We will denote by $Inj(J)$ the set of injective elements of $J$ \cite{fgm}. \end{apartado}

\begin{apartado}\label{classical_order} A Jordan algebra $Q$ containing $J$ as a subalgebra is a \emph{classical algebra of quotients of $J$}, or an \emph{algebra of fractions} of $J$ (and $J$ is a \emph{classical order in $Q$}) if all injective elements of $J$ are invertible in $Q$ and  ${\cal D}_J(q)\cap Inj(J)\neq\emptyset$ for all $q\in Q$.  That is, classical algebras of quotients are $S$-algebras of quotients (or  algebras of fractions relative to $S$) for $S=Inj(J)$. Moreover, they are algebras of quotients (in the sense of \ref{aq}) \cite[Examples 2.3.5]{densos}.
\end{apartado}

\begin{apartado}\label{quot FC} The notions of classical order and that of algebra of quotients are equivalent for subalgebras of unital Jordan algebras of finite capacity. Indeed, if $J$ is a subalgebra of a unital Jordan algebra of finite capacity $Q$, then it was proved in \cite[Lemma 2.3]{mp-localorders} that $J$ is a classical order in $Q$ if and only if $Q$ is an algebra of quotients of $J$.
\end{apartado}

\begin{apartado} A nondegenerate Jordan algebra is {\it Goldie} if and only if it is nonsingular and has finite uniform dimension. Nondegenerate Goldie Jordan algebras were characterized in \cite[Theorem 9.3]{fgm}. We only recall here that although nondegenerate Goldie Jordan algebras satisfy the property that  essential inner ideals contain injective elements, such Jordan algebras are not necessarily Goldie in contrast to the analogous assertion for associative rings. Indeed it suffices to consider any Jordan algebra of a quadratic form on a vector space with infinite dimensional totally isotropic subspaces.
\end{apartado}

\begin{teorema}\label{th-FGM goldie LC JA}{\rm  \cite[Theorem 10.2]{fgm}}  A Jordan algebra $J$ is a classical order in a nondegenerate unital Jordan algebra $Q$ with finite capacity if and only if it is nondegenerate and satisfies the property: an inner ideal $K$ of $J$ is essential if and only if $K$ contains an injective element. Moreover, $Q$ is simple if and only if $J$ is prime.    \end{teorema}

\begin{apartado}\label{def LC element JA}  Jordan algebras satisfying the    equivalent conditions of  \cite[Theorem 10.2]{fgm} are named {\sl Lesieur-Croisot} Jordan algebras, LC-algebras for short, in \cite[Definition 3.2]{mt-lc}. An element $a$  of a Jordan algebra $J$ is said to be a {\sl Lesieur-Croisot element} (LC-element, for short) if the local algebra of $J$ at $a$ is a LC-algebra. The set  $LC(J)$ of LC-elements of a Jordan algebra $J$  is   an ideal   if the Jordan algebra $J$  is nondegenerate \cite[Theorem 5.13]{mt-lc}. It follows straightforwardly from \cite[Theorem 9.3]{fgm} that nondegenerate Goldie Jordan algebras are LC-algebras. LC-Jordan algebras were studied by Montaner and Tocón in \cite{mt-lc} (see also \cite{lc-pr}).
  \end{apartado}

 \section{The singular ideal of semiprime associative pairs.}

In this section we study the set  of elements of a semiprime associative pair with essential left (or right) annihilator.

\begin{apartado}\label{assoc-imbedding}  Let $A=(A^+,A^-)$ be an associative pair. Following \cite{loos-jp} we denote by  $({\cal E}, e)$ the {\it standard imbedding } of $A$, a unital associative algebra having an idempotent $e$ which produces a Peirce decomposition ${\cal E}={\cal E}_{11}\oplus {\cal E}_{12}\oplus {\cal E}_{21}\oplus {\cal E}_{22}$ where ${\cal E}_{ij}=e_i{\cal E}e_j$ with $e_1=e$ and $e_2=1-e$, such that $A=({\cal E}_{12},{\cal E}_{21})$. Recall that ${\cal E}_{11}$ is spanned by the idempotent $e$ and all the elements $x_{12}y_{21}$ where $x_{12}\in{\cal E}_{12}$, $y_{21}\in {\cal E}_{21}$. Similarly ${\cal E}_{22}$ is spanned by the element $1-e$ and all the elements $y_{21}x_{12}$. We will denote by  ${\cal A}$ the subalgebra of ${\cal E}$ generated by the elements $x_{12}\in{\cal E}_{12}$ and $y_{21}\in {\cal E}_{21}$.
We will say that ${\cal A}$ is the {\it associative envelope} of $A$. Note that ${\cal A}$ is an essential ideal of ${\cal E}$ and $A=({\cal A}_{12},{\cal A}_{21})$ where ${\cal A}_{ij}=e_i{\cal A}e_j$, $i,j=1,2$.
\end{apartado}

\begin{remark} Notation for the  associative envelope ${\cal E}$ and the standard imbedding ${\cal A}$   of an  associative pair $A$, as used here, follows that considered in \cite{mp-lancaster}. We refer the reader to  \cite[Remark 2.1]{mp-lancaster} for discussion about the different notation and terminology used for  ${\cal E}$ and   ${\cal A}$. See, for instance, \cite{goldie pares, ft, gs comm alg32,loos-1995}.
\end{remark}

\begin{apartado} Given an associative pair $A=(A^+,A^-)$, we will say that a $\Phi$-submodule  $L$ of $A^\sigma $  is a {\sl left ideal} of $A$ if $A^\sigma  A^{-\sigma  }L\subseteq L$, $\sigma =\pm$. Right ideals are defined similarly. An {\sl  ideal} $I=(^+,I^-)$ of $A$ is a pair of two-sided (i.e. left and right) ideals $I^\sigma\subseteq A^\sigma$, such that $A^\sigma I^{-\sigma}A^\sigma\subseteq I^\sigma$, $\sigma=\pm$.
\end{apartado}

\begin{apartado} A left ideal $L$ of $A$, contained in $A^\sigma $, is {\it   essential} if it has   nonzero intersection with any nonzero left ideal of $A$ contained in $A^\sigma $. Essential right ideals are defined similarly.
\end{apartado}

\begin{apartado} An associative pair $A$ is {\sl semiprime} if and only if $I^\sigma A^{-\sigma}I^\sigma=0$, $\sigma=\pm$, implies $I=0$ for any ideal $I$ of $A$ and {\sl prime} if and only if $I^\sigma A^{-\sigma}J^\sigma=0$, $\sigma=\pm$, implies $I=0$ or $J=0$ for any ideals $I$ and $J$ of $A$. An associative pair $A$ is (semi)prime if and only if so is its standard imbedding ${\cal E}$ if and only if so is its associative envelope ${\cal A}$    \cite[Proposition 4.2]{goldie pares}.
\end{apartado}

\begin{apartado}  Let  $A=(A^+,A^-)$ be an associative pair. Any element    $a\in A^{-\sigma}$, endows the $\Phi$-module $A^\sigma$ with a multiplication $x\cdot_a y=xay$  for any $x,y\in A^\sigma$, and defines an associative  $\Phi$-algebra $(A^\sigma)^{(a)}$ called the $a$-{\sl homotope algebra} of $A$. The {\sl local algebra $A_a^\sigma$ of $A$ at $a$} is defined as the quotient of the $a$-homotope  algebra $(A^\sigma)^{(a)}$ by the ideal $ Ker \, a=\{x\in A^\sigma\mid axa= 0\}$. Moreover, for all $a\in A^{-\sigma}$, we have $A^\sigma_a\cong {\cal E}_a\cong {\cal A}_a$ \cite[Proposition 3.4(3)]{ft}.
\end{apartado}

\begin{apartado}\label{PI elements AP}  An element $a\in A^{-\sigma}$ of an associative pair $A$ is a PI-{\sl element} if the local algebra $A_a^\sigma$ of $A$ at $a$ satisfies a polynomial identity. We write $PI(A)=(PI(A^+),PI(A^-))$,  where $PI(A^\sigma)$ is the set of PI-elements of $A^\sigma$.  The set of PI-elements of any associative pair was proved to be an ideal in \cite[Proposition 1.6]{pi-i}. Moreover, if $A$ is semiprime $PI(A)=  PI({\cal E})\cap A$ \cite[Proposition 2.3]{mp-lancaster}.
\end{apartado}

\begin{apartado}  Let $A$ be an associative pair. For any subset $X\subseteq A^\sigma $, $\sigma =\pm$, the sets \begin{align*} &lann_A(X)=\{b\in A^{-\sigma }\mid bXA^{-\sigma }=A^\sigma bX=0\}\\ &rann_A(X)=\{b\in A^{-\sigma }\mid XbA^\sigma =A^{-\sigma }Xb=0\} \end{align*} are the  {\sl left} and {\sl right annihilators} of $X$ in $A$ and are left and right ideals of $A$ respectively. If $A$ is semiprime, then \begin{align*} &lann_A(X)=\{b\in A^{-\sigma }\mid  A^\sigma bX=0\}\\ &rann_A(X)=\{b\in A^{-\sigma }\mid  A^{-\sigma }Xb=0\}. \end{align*}  If $X=\{x\}$ we will  denote by $lann_A(x)$ and $rann_A(x)$ the left and right annihilators of the element  $x$. The {\sl annihilator} of $X$ is $ann_A(X)=lann_A(X)\cap rann_A(X)$. If $A$ is semiprime, for any ideal $I=(I^+,I^-)$ of $A$, $ann_A(I)=(ann_A(I^+),ann_A(I^-))$, where $ann_A(I^\sigma)=\{x\in A^{-\sigma}\mid xI^\sigma x=0\}$, is an ideal of $A$ \cite[Proposition 2.2]{goldie pares}.
\end{apartado}

\begin{apartado}  Let $A$ be a semiprime associative pair.  Then   $Z_l(A)=(Z_l(A^+),Z_l(A^-))$, where $Z_l(A^\sigma )=\{a\in A^\sigma \mid lann_A(a)\ \hbox{is an essential left ideal of $A$}\}$,  is an ideal of $A$ called  the {\sl left singular ideal } of $A$   \cite[Theorem 3.1]{goldie pares}. Analogously we  define $Z_r(A)$ to be the {\sl right singular ideal} of $A$. We will say that $A$ is {\sl left nonsingular} if $Z_l(A)=0$. {\sl Right nonsingularity} is defined similarly. An associative pair $A$ is said to be {\sl nonsingular} if it is both left and right
nonsingular.
\end{apartado}

\begin{proposicion}\label{ideal singular pares asociativos} Let $A$ be a semiprime associative pair, and $a\in A^{-\sigma}$. Then:
\begin{enumerate}
\item[(i)] If $\overline{x}=x+Ker\, a\in Z_l(A_a^\sigma)$, then $axa\in Z_l(A^{-\sigma})$.
    \item[(ii)] If $a\in Z_l(A^{-\sigma})$, then $Z_l(A_a^\sigma)=A_a^\sigma$.
\item[(iii)] $\overline{Z_l(A^\sigma)}=(Z_l(A^\sigma)+Ker\, a)/Ker\, a\subseteq Z_l(A^\sigma_a)$.
Thus  $Z_l(A^\sigma_a)=0$ implies that  $a\in ann_A(Z_l(A^{\sigma}))$.
\end{enumerate}
\end{proposicion}

\begin{proof} (i) and (ii) follow from  \cite[Proposition 5.3(i)-(ii)]{goldie pares}. To prove (iii), take $x\in Z_l(A^\sigma) $ with $ \overline{x}=x+Ker\, a\neq0$, so that $axa\neq0$ and let $ \overline{L}$ be a nonzero left ideal of $A^\sigma_a$, where $-$ denotes the projection of $(A^\sigma)^{(a)}$ onto $A^\sigma_a$. We can assume that $L$ is a nonzero left ideal of $A$, contained in $A^\sigma$ (see the proof of \cite[Lemma 6.3]{fgm}). Moreover, as $Z_l(A)$ is an ideal of $A$, and therefore $axa\in Z_l(A)$, by essentiality of $lann_A(axa)$  in $A$, we can assume that $L\subseteq lann_A(axa)$.

Since $ \overline{L}\neq 0$, we have $aLa\neq0 $ and then, by semiprimeness of $A$,   $aLaA^\sigma aLa\neq0$. Hence $0\neq LaA^\sigma aL\subseteq lann_A(axa)$. Take now $l_1,l_2\in L$, $r\in A^\sigma$  with $ al_1aral_2a\neq 0$. Then $0\neq l_1aral_2 \in LaA^\sigma aL\subseteq L \cap lann_A(axa)$. Therefore  we have $a l_1aral_2axa=0$ which implies that $\overline{l_1}\cdot\overline{r}\cdot\overline{l_2}\cdot\overline{x}=\overline{0}$, hence $\overline{l_1}\cdot\overline{r}\cdot\overline{l_2}\in \overline{L}\cap lann_{A^\sigma_a}(\overline{x})$ and $\overline{l_1}\cdot\overline{r}\cdot\overline{l_2}\neq0$ since $al_1aral_2a\neq0$, implying that $ \overline{x}\in Z_l(A^\sigma_a)$.

Finally, if $Z_l(A^\sigma_a)=0$, then we have $aZ_l(A^\sigma)a=0$ and therefore, by \cite[Proposition 2.2(ii)]{goldie pares}, $a\in ann_A(Z_l(A^{\sigma}))$.
\end{proof}

\begin{apartado} The {\sl left Goldie (or uniform) dimension} of a left ideal $L$ of an associative pair $A$ is a nonnegative integer $n$ such that $L$ contains a direct sum of $n$ nonzero left ideals and any direct sum of nonzero left ideals contained in $L$ has at most $n$   nonzero left ideals. The {\sl left Goldie (or uniform) dimension} of an element $a\in A^\sigma$ is the left Goldie dimension of the principal left ideal $A^\sigma A^{-\sigma}a$  generated by $a$.  A left ideal $L$ of a semiprime associative pair  $A$, contained in $A^\sigma$, has {\sl finite left Goldie (or uniform) dimension} if it contains no infinite direct sums of nonzero left ideals of $A$, contained in $A^\sigma$. If any element of $A$ has finite left Goldie (or uniform) dimension, then it is said that $A$ has {\sl finite left local Goldie (or uniform) dimension}. Following \cite{gs comm alg32}  we write $I_l(A)=(I_l(A^+),I_l(A^-))$, where $I_l(A^\sigma)$ denotes the set of elements of $A^\sigma$ having finite left Goldie (or uniform) dimension. Similarly, we denote the set of elements of $A $ with finite right Goldie (or uniform) dimension as $I_r(A) $.
\end{apartado}

\begin{proposicion}\label{assoc pairs ideal I} Let $A$ be a semiprime associative pair. If $Z_l(A )\cap Z_r(A )=0$, then $I_l(A) \cap I_r(A)$ is an ideal of $A$ contained in $ann_A(Z_l(A ) + Z_r(A ))$.
\end{proposicion}

\begin{proof} The proof of \cite[Proposition 4.6]{mt-lc} can be used here with only slight modifications for adapting it to the pair setting. Indeed, it suffices to replace    references  for associative algebras \cite[Lemma 4.5(i)-(ii)]{mt-lc} by Proposition \ref{ideal singular pares asociativos}(i)-(iii), and use \cite[Proposition 3.2]{goldie pares} and \cite[Proposition 3.3(iv)]{gs comm alg32}.
\end{proof}

\section{PI-elements of  special Jordan pairs.}

PI-elements play an important role in the local LC-theory of Jordan algebras developed in \cite{mt-lc}. To tackle the local LC-theory for Jordan pairs, we study in this section the relationship between  the PI-ideal of a strongly prime Jordan pair of hermitian type and that of its $\ast$-tight associative pair envelope, thus extending to Jordan pairs results originally given for special Jordan algebras in  \cite[Theorem 6.5]{pi-i}.

\begin{lema}\label{amples} Let  $V$ be a special Jordan pair, subpair of $H(A,\ast)$  for  an associative pair $A$ with involution $\ast$,                                                             and let  ${\cal H}(X)$
be a hermitian ideal of the free  special Jordan pair $FSJP[X]$.
Then:
\begin{enumerate}\item[(i)] for any ideal  $I$ of $V$ the $n$-tads
$$\{\ldots, {\cal
H}(V)^\sigma,I^{-\sigma},\big({\cal H}(V)^{(2\rangle}\big)^\sigma,\big({\cal H}(V)^{(2\rangle}\big)^{-\sigma},{\cal
H}(V)^{ \sigma},\ldots\}$$  are in $I$ for all odd  $n$.
\item[(ii)] if $I={\cal H}(V)^{(2\rangle}$, for any   $a\in V^\sigma$
the $n$-tads $\{I^\sigma,\ldots,I^{-\sigma},a\}$  and
$\{I^{-\sigma},\ldots,I^{-\sigma},a,I^{-\sigma}\}$ are in $I$  for
all odd $n$.
\end{enumerate}
\end{lema}

\begin{proof}   Take  $s\in \big({\cal H}(V\big)^{\langle1\rangle})^\sigma$,
$t\in \big({\cal H}(V)^{\langle1\rangle}\big)^{-\sigma}$, $r\in \big({\cal H}(V)^{(2\rangle}\big)^{-\sigma}$ and   variables   $y\not\in X^{-\sigma}$,  $z\not\in X^{\sigma}$  (where    $X$  is considered to be a polarized
set of variables).  Note that, since ${\cal H} (V)^{(2\rangle}\subseteq {\cal H} (V)^{\langle1\rangle}\subseteq {\cal H} (V)^{(1)}$, we can write $r=Q_pq\in {\cal H} (V)^{\langle1\rangle}  = Q_{{\cal H} (V)}{\cal H} (V)$. Now:
\begin{align*}  \{\ldots,{\cal H }(X)^\sigma,y,Q_st,  \big({\cal H}(V)^{(2\rangle}\big)^{-\sigma},& \ldots\}=\\
=& \{\ldots,{\cal H }(X)^\sigma,\{y,s,t\},s,\big({\cal H}(V)^{(2\rangle}\big)^{-\sigma},\ldots\}-\\
  -&\{\ldots,{\cal H}(X)^\sigma, t,Q_sy,\big({\cal H}(V)^{(2\rangle}\big)^{-\sigma},\ldots\}\end{align*} and
\begin{align*}
   \{\ldots, {\cal H }(X)^\sigma,y,Q_zQ_ts,r, &{\cal
H}(X)^{\sigma}, \ldots\}=\\
&=\{\ldots ,{\cal H}(X)^{\sigma}, \{y,z,Q_ts\},z,Q_pq, {\cal H}(X)^{\sigma},\ldots\}-\\
&-\{\ldots, {\cal H}(X)^{\sigma},Q_ts,Q_zy,Q_pq,{\cal H}(X)^{\sigma},\ldots\}=\\
&=\{\ldots,{\cal H}(X)^{\sigma},\{y,z,Q_ts\},\{z,p,q\},p,{\cal H}(X)^{\sigma},\ldots\}-\\
&-\{\ldots,{\cal H}(X)^{\sigma},\{y,z,Q_ts\},q,Q_pz,{\cal H}(X)^{\sigma},\ldots\}-\\
&-\{\ldots,{\cal H}(X)^{\sigma},Q_ts,\{Q_zy,p,q\},p,{\cal H}(X)^{\sigma},\ldots\}+\\
&+ \{\ldots,{\cal H}(X)^{\sigma},Q_ts,q,Q_pQ_zy,{\cal H}(X)^{\sigma},\ldots\}
\end{align*}
are sums of $n$-tads ($n$ odd) elements of  ${\cal H}(X\cup\{y, z\})$, so they are Jordan polynomials in $FSJP(X\cup\{y,z \})$, and the proof of \cite[Lemma 2.1]{acm} applies here yielding (i).

Finally, we note  that  if  $n<5$, then   (ii)  is trivial, since $I$ is an ideal of $V$, and for  $n\geq 5$ it follows from (i) above.
 \end{proof}

The next result states that local algebras of ample subpairs of associative pairs with involution retain ampleness in the set of symmetric elements of the local algebra with the induced involution.

\begin{corolario}\label{corolario-amples} Let $V=H_0(A,\ast)$  be an ample subpair of an associative pair $A$ with an involution $\ast$.
Then for any  $a\in V^{-\sigma}$ the local algebra  $V^\sigma_a$ is an ample subspace of $H(A^\sigma_a,\ast)$,  $\sigma=\pm$.\end{corolario}

\begin{proof} The proof is straightforward from Lemma \ref{amples} (see \cite[Lemma 2.3]{acm} for the algebra case).\end{proof}

The proof of the next result follows that of \cite[Theorem 6.5]{pi-i}.  Here however we need  to tackle the fact that, as   noted in \ref{derived ideals}, products of ideals of Jordan pairs are not again ideals, merely semi-ideals. We will use Lemma \ref{amples} to overcome this additional difficulty on the generation of powers of ideals.

\begin{teorema}\label{relacion pi-ideal hermitiano} Let   $V$  be a strongly prime Jordan pair with a    $\ast$-tight associative envelope $A$. If $V$
is of hermitian type, then $PI(V)=PI(A)\cap V$. \end{teorema}

\begin{proof} The containment $PI(A)\cap V\subseteq PI(V)$ was proved in   \cite[Lemma~6.4(b)]{pi-i} for  special Jordan systems (including Jordan pairs).

To prove the reverse inclusion take  $a\in PI(V^\sigma)$ and suppose ${\cal H}(V)\neq0$ for a  hermitian ideal  ${\cal H}(X)$  of  $FSJP[X]$. Set $I={\cal H}(V)^{(2\rangle}$  (see \ref{derived ideals}) and denote by $S$ the Jordan subpair of $V$ generated by $I$ and the element $a$,  by  $R$  the associative subpair of $A$ generated by $S$ and by  $B$ the subpair of $R$ generated by  $I$. Then $S$ is strongly prime, since it contains the nonzero ideal $I$ of $V$ (the proof of \cite[Lemma 2.4]{acm} applies here), and it follows from Lemma \ref{amples}  that $S=H_0(R,\ast)$ and $I=H_0(B,\ast)$. Moreover, by  \cite[Lemma 2.32]{da-hermitian}, and as a result of the correspondence between Jordan pairs and polarized triple systems (see \cite[Section 5]{da-hermitian} for further details), we have that  $B$ is $\ast$-tight over $I$.

Now by \cite[Jordan Memory Theorem 4.2]{mc-martindale} we have $ I=H_0(B,\ast)\lhd V\subseteq H(Q_s(B),\ast)$,   where $Q_s(B)$ is the pair of Martindale symmetric quotients of $B$. Moreover we have $B\subseteq R\subseteq Q_s(B)$. Hence  $R$ is semiprime  by\cite[Lemma 2.13(ii)]{gs comm alg32}.

Next, since $a\in PI(V)$, the local algebra $S^{-\sigma}_a$ which is a subalgebra of the PI-algebra $V^{-\sigma}_a$, is itself a Jordan PI-algebra, and since by Corollary \ref{corolario-amples} we have $S^{-\sigma}_a=H_0(R^{-\sigma}_a,\ast)$, it follows from \cite[Corollary 7.4.13]{rowen-pi}  that  $R^{-\sigma}_a$ is an associative PI-algebra.

Consider now the standard imbedding ${\cal E}_R$ of the associative pair $R$  (see \ref{assoc-imbedding}). Since $({\cal E }_R) \cong R^{-\sigma}_a$ by \cite[Proposition 3.4(3)]{ft} , we have that the associative algebra   ${\cal E}_R$  satisfies a GPI. (Note also that $a\in PI(R)=PI({\cal E}_R)\cap R$  by \cite[Proposition 2.3]{mp-lancaster}.) But then, as a result of  ${\cal E}_A\subseteq Q_s( {\cal E}_A)=Q_s({\cal
E}_R)$, we   have that the standard imbedding  ${\cal E}_A$ of the associative pair $A$ satisfies the same GPI as ${\cal E}_R$ does. Therefore, we finally obtain that $a\in PI(A^\sigma )$ since we have $({\cal E}_A)_a \cong A_a^{-\sigma}$.
 \end{proof}

\section{Lesieur-Croisot elements of nondegenerate Jordan pairs }

This section is devoted to study the set of LC-elements of a Jordan pair, concretely to prove that this set is an ideal if $V$  is nondegenerate. The proofs of the results  in this section that closely follow those given in \cite{mt-lc} for Jordan algebras are  just outlined. Yet, all necessary references relative to (associative and Jordan) pairs are however provided.

\begin{apartado}\label{def-LCelement JP} Let $V$ be a Jordan pair. We will say that an element $x\in V^\sigma$, $\sigma=\pm$ is a {\sl Lesieur Croisot element} of $V$ (and LC-element for short) if the local algebra $V^{-\sigma}_x$ of $V$ at $x$ is a Lesieur-Croisot Jordan algebra (see \ref{def LC element JA}).  We will denote by  $LC(V)=(LC(V^+),LC(V^-))$ the set of LC-elements of a Jordan pair $V$.
\end{apartado}

We start  stating that the set of LC-elements of any strongly prime Jordan pair with nonzero PI-elements  coincides with its PI-ideal.

\begin{proposicion}\label{LC-if-PIelements} Let $V$ be a strongly prime Jordan pair having nonzero PI-elements. Then $LC(V)=PI(V)$, and in particular, $LC(V)$ is an ideal of $V$.
\end{proposicion}
\begin{proof} The proof of \cite[Proposition 3.3]{mt-lc} works here. \end{proof}

\begin{proposicion}\label{lc-fv} Let $V$ a strongly prime Jordan pair such that $PI(V)=0$ and $LC(V)\neq0$. Then $V$ is nonsingular and $LC(V)=F(V)$, where $F(V)=(F(V^+),F(V^-))$  with $F(V^\sigma)=\{x\in V^\sigma\mid udim(V_x^{-\sigma})<\infty\}$.
\end{proposicion} \begin{proof} Following the  proof of \cite[Proposition 4.2]{mt-lc}, it suffices to replace \cite[Proposition 6.4]{fgm}
by Proposition \ref{JP ideal singular prop}(i) and  Corollary \ref{corolario ideal singular}(i), and recall that local algebras of nonsingular Jordan pairs inherit nonsingularity by Corollary \ref{corolario ideal singular}(ii).
\end{proof}

Next we prove a pair analogue  of a result contained into the proof of \cite[Theorem 6.5]{pi-i} (see also \cite[Lemma 4.3]{mt-lc}). To prove the following lemma we  make use of a result from \cite{ACMcM_amples} on the $\ast$-tightness of any $\ast$-prime associative system over   its nonzero ample subsystems together to some other intermediate results appearing in the proof of Theorem \ref{relacion pi-ideal hermitiano}.

\begin{lema}\label{amples-tightness JP} Let $V$ be a strongly prime Jordan pair such that $PI(V)=0$ and let the associative pair $A$ be a $\ast$-tight associative envelope of $V$. Then for each $a\in V $, the subpair $S$ of $V$ generated by the ideal $I={\cal H}(V)^{(2\rangle}$, where ${\cal H}(V)\neq0$ for some nonzero hermitian ideal ${\cal H}(X)$ and the element $a$ is strongly prime of hermitian type. Moreover $S=H_0(R,\ast)$, where $R$, the associative subpair of $A$ generated by $I\cup\{a\}$,  is a $\ast$-tight associative envelope of $S$.
\end{lema}
\begin{proof} It follows from the proof of Theorem  \ref{relacion pi-ideal hermitiano}   that $S$ is a strongly prime Jordan pair and an ample subpair $S=H_0(R,\ast)$ of $R$, where $R$ denotes the associative subpair of $A$ generated by $I$ and the element $a$. Now, since  $R$ is semiprime (indeed, it is $\ast$-prime), the $\ast$-tightness of $R$ over $S$ follows from \cite[Proposition 3.1]{ACMcM_amples}.
  \end{proof}

\begin{teorema}\label{Sprime PI-less}  Let $V$ be a strongly prime Jordan pair with $PI(V)=0$  and let the associative pair $A$ be a $\ast$-tight associative envelope of $V$. Then  we have $ F(V)=I_l(A)\cap V=I_r(A)\cap V$, where $I_l(A)=(I_l(A^+),I_l(A^-))$ with $I_l(A^\sigma)=\{a\in A^\sigma \mid udim_l(A^{-\sigma}_a)<\infty\} $ and $I_r(A)=(I_r(A^+),I_r(A^-))$ with $I_r(A^\sigma)=\{a\in A^\sigma \mid udim_r(A^{-\sigma}_a)<\infty\} $. Moreover if $A$ is prime, then $udim(V^{-\sigma}_a)=udim_l(A^{-\sigma}_a)=udim_r(A^{-\sigma}_a)$ for all $a\in F(V^\sigma)$, $\sigma=\pm$.
\end{teorema}
\begin{proof}  Let $a\in V^\sigma$. We first note that   it suffices to prove  the left version of the claimed statement
as then the right one will follow symmetrically, that is, it suffices to prove that the local algebra  $V^{-\sigma}_a$ of $V$ at $a$  has finite uniform dimension if and only if so has $A^{-\sigma}_a$, and that both  uniform dimensions coincide if $A$ is prime. It is straightforward to check that the  arguments given in the proof of \cite[Theorem 4.4]{mt-lc} work  here applied now to Jordan pairs. Indeed, it suffices to replace \cite[Lemma 4.3]{mt-lc} by Lemma \ref{amples-tightness JP} and recall that the set $PI(A)$ of PI-elements of an associative pair $A$ is an ideal of $A$  by \cite[Proposition 1.6]{pi-i}. Finally  we recall  the relationship between the set of PI-elements of any strongly prime Jordan pair of hermitian type and those of any of its $\ast$-tight associative envelopes   given in Proposition \ref{relacion pi-ideal hermitiano} and  also that between local algebras  of associative pairs  and those of their standard imbeddings    (see \ref{PI elements AP}).
\end{proof}

\begin{proposicion}\label{relacion ideales singulares} Let $V$ be a strongly prime PI-less Jordan pair and let the associative pair $A$ be a $\ast$-tight  associative envelope of $V$. If $V$ is nonsingular, then  $Z_l(A)\cap V=0$.
\end{proposicion}
\begin{proof}  We first note that $V$ is special of hermitian type since it is PI-less. Take a hermitian ideal ${\cal H}(X)$ with ${\cal H}(V)\neq0$, and let $R$ be the associative subpair of $A$ generated by ${\cal H}(V)$. Then ${\cal H}(V)=H_0(R,\ast)$ is an ample subpair of $H(R,\ast)$ and $R$ is $\ast$-tight over ${\cal H}(V)$ \cite[Lemma 2.32]{da-hermitian}. Therefore $R$ is $\ast$-prime. Moreover since $R\subseteq A \subseteq Q_s(R)$, where $Q_s(R)$ denotes the Martindale pair of symmetric quotients of $R$ we have $Z_l(A)\cap R=Z_l(R)$ \cite[Proposition 2.14(ii)]{gs comm alg32}. Thus since by  Corollary \ref{corolario ideal singular}(v) the ideal ${\cal H}(V)$ is nonsingular, we may replace $V$ by ${\cal H}(V)$  and assume that $V=H_0(A,\ast)$ is an ample subpair of $A$.

Take now $a\in Z_l(A)\cap V^\sigma$. By Corollary \ref{corolario-amples} we have $V^{-\sigma}_a=H_0(A^{-\sigma}_a,\ast)$ and, since $A^{-\sigma}_a$ is $\ast$-prime, it follows from  \cite[Theorem~1.2]{ACMcM_amples}
that $A^{-\sigma}_a$ is $\ast$-tight over $V^{-\sigma}_a$. Therefore, by \cite[Theorem 6.14]{fgm}, $Z_l(A^{-\sigma}_a)\cap V^{-\sigma}_a=\Theta(V^{-\sigma}_a)$, where $Z_l(A^{-\sigma}_a)= A^{-\sigma}_a $ by Proposition \ref{ideal singular pares asociativos}(ii). Hence we have $\Theta(V^{-\sigma}_a)=V^{-\sigma}_a$ which implies  $Q_aV^{-\sigma}\subseteq \Theta(V)=0$ by Proposition \ref{JP ideal singular prop}(ii), and we finally obtain $a=0$ by nondegeneracy of $V$.
\end{proof}

\begin{teorema}\label{th LC ideal SP-JP} Let $V$ be a strongly prime Jordan pair. Then $LC(V)$ is an ideal of $V$.
\end{teorema}
\begin{proof} Obviously we can assume  that  $LC(V)\neq0$. If $PI(V)\neq0$ it was already  proved in Proposition \ref{LC-if-PIelements}  that $LC(V)$ is an ideal of $V$, we can therefore suppose that $V$ is PI-less. Let $A$ be a $\ast$-tight associative envelope of $V$,   we have that $V$ is nonsingular, and that $LC(V)=F(V)$ by Proposition  \ref{lc-fv}.
 Moreover   $(Z_l(A)\cap Z_r(A))\cap V=0$ by Proposition \ref{relacion ideales singulares}. Thus, since $Z_l(A)\cap Z_r(A)=Z_l(A)\cap Z_l(A)^\ast$ is a  $\ast$-ideal of $A$, we get $Z_l(A)\cap Z_r(A)=0$ by $\ast$-tightness of $A$ over $V$. Now, by Proposition \ref{assoc pairs ideal I},  $I_l(A)\cap I_r(A)=I_l(A)\cap I_l(A)^\ast$  is a  $\ast$-ideal of $A$  contained  in $ ann_A(Z_l(A)+Z_r(A))$. (The proof that
$I_l(A)\cap I_r(A)=I_l(A)\cap I_l(A)^\ast$  is an $\ast$-ideal of $A$ works as in   \cite[Proposition 4.6]{mt-lc}.)
    Hence $LC(V)$ is an ideal of $V$ by  Theorem \ref{Sprime PI-less}.
\end{proof}

Finally, to extend the above result to nondegenerate Jordan pairs it suffices to restate adapted to pairs the results of \cite[Section 5]{mt-lc} to prove that the set of LC-elements of a nondegenerate Jordan algebra is an ideal. We anyhow recall here some of the most relevant ideas needed to prove the required results.

\begin{apartado}
Let $V$ be a nondegenerate Jordan pair. An ideal $I$ of $V$ is said to be {\sl semi-uniform}
if there exist prime ideals $P_1,\ldots,P_n$ of $V$ such that $P_1\cap \cdots \cap  P_n\subseteq ann_V(I)$.     The Jordan pair $V$ is said to be {\sl semi-uniform} if it is a semi-uniform ideal.  For any semi-uniform ideal $I$ of a nondegenerate Jordan pair $V$ there exists a  minimal set of prime ideals ${\cal P}=\{P_1,\ldots,P_n\}$ of $V$ such that  $P_1\cap \cdots P_n\subseteq ann_V(I)$, in fact it holds that $P_1\cap \cdots \cap  P_n=ann_V(I)$. This set of prime ideals, named {\sl the minimal set of prime ideals associated to $I$} is unique and can be  shown to be equal to ${\cal P}=\{P\vartriangleleft V\mid P \ \hbox{minimal prime and}\   ann_V(I)\subseteq P \}$ \cite[Lemma 5.3]{mt-lc}.
\end{apartado}

\begin{apartado} For any nondegenerate Jordan pair $V$ being semi-uniform is equivalent to being an essential subdirect sum of finitely many strongly prime Jordan pairs. Recall that (see  \cite[p. 448]{fgm} or \cite{fg-lattices}) an {\sl essential subdirect product} is a subdirect product which contains an essential ideal of the full direct product. If $V$ is contained in the direct sum   $\bigoplus V_\alpha$, then $V$ is called an {\sl essential subdirect sum}.
\end{apartado}

\begin{proposicion}
 Let $V$ be a nondegenerate Jordan pair. Then the following conditions are equivalent:
 \begin{enumerate}
 \item[(i)] $V$ is semi-uniform.
 \item[(ii)]  $V$ has no infinite direct sums of ideals.
  \item[(iii)] $V$ has the acc on annihilator ideals.
   \item[(iv)] $V$ is an essential subdirect sum of finitely many strongly prime Jordan pairs.
 \end{enumerate}
 \end{proposicion}
 \begin{proof} The proof of \cite[Proposition  5.6]{mt-lc} works here. \end{proof}

 Local algebras of essential subdirect sums of nondegenerate Jordan pairs are subdirect sums of nondegenerate Jordan algebras which are obtained as local algebras of the essential subdirect factors.  Moreover local algebras of nondegenerate Jordan pairs inherit the property of being semi-uniform.

\begin{lema}
 Let $V$ be a nondegenerate Jordan pair and let $a\in V^{-\sigma}$. If $V$ is an  essential subdirect sum of Jordan pairs $V_1,\ldots, V_n$ and $a=a_1+\cdots+a_n$, with $a_i\in V_i^{-\sigma}$, then the local algebra $V^\sigma_a$ of $V$ at $a$ is an
 essential subdirect sum of the local lgebras $(V_1^\sigma)_{a_1},\ldots,(V_n^\sigma)_{a_n} $.
  \end{lema}
 \begin{proof} The proof of \cite[Lemma 5.9]{mt-lc} works here. \end{proof}

  \begin{proposicion}
 Let $V$ be a nondegenerate Jordan pair and let $a\in V^{-\sigma}$.
 \begin{enumerate}
 \item[(i)] If $V$ is semi-uniform, then $V^\sigma_a$ is semi-uniform.
 \item[(ii)]  $V^\sigma_a$ is semi-uniform if and only if the ideal $id_V(a)$ generated by $a$ is semi-uniform.
 \end{enumerate}
 \end{proposicion}
  \begin{proof} The proof of \cite[Proposition  5.10]{mt-lc} works here. \end{proof}

\begin{apartado}  Let $V$ be a nondegenerate Jordan pair $V$. Then  the set    $SU(V)=(SU(V^+),SU(V^-))$ where
$SU(V^\sigma)=\{a\in V^\sigma\mid V^{-\sigma}_a\ \hbox{is semi-uniform}\}$ is an ideal of $V$ (this is proved as \cite[Proposition  5.11]{mt-lc}). Note that  for any nonzero LC-element $a\in LC(V^{-\sigma})$ of a  nondegenerate Jordan pair $V$  that $a\in SU(V^{-\sigma})$ by \cite[Theorem 10.2(ii), 10.3]{fgm} and \cite[Proposition 5.6]{mt-lc} applied to the local algebra $V^\sigma_a$ of $V$ at $a$, hence  $LC(V)\subseteq SU(V)$.
\end{apartado}

We finally obtain the desired result on the set of LC-elements of nondegenerate Jordan pairs.

\begin{teorema}  Let $V$ be a nondegenerate Jordan pair. Then the set $LC(V)=(LC(V^+),LC(V^-))$ where
$LC(V^\sigma)=\{a\in V^\sigma\mid V^{-\sigma}_a\ \hbox{is an LC-algebra}\}$ is an ideal of $V$.
\end{teorema}
\begin{proof}  The proof of \cite[Theorem 5.13]{mt-lc} works here to reduce the nondegenerate case to that of strongly prime Jordan, which in turn follows from Theorem \ref{th LC ideal SP-JP}.
\end{proof}


 \addcontentsline{toc}{chapter}{\protect\numberline{}  Bibliograf\'{\i}a.}



\begin{thebibliography}{11111111}
\bibitem[ACM]{acm}  J. A. Anquela, T. Cort\'{e}s, F. Montaner,  Local inheritance in Jordan algebras,    Arch. Math.
64
  (1995)   393-401.

\bibitem[ACMcM]{ACMcM_amples}  J. A. Anquela, T. Cort\'{e}s, K. McCrimmon, F. Montaner,   Strong primeness of Hermitian Jordan systems,   J. Algebra
198
  (1997)  311-326.

  \bibitem[DA]{da-hermitian}  A. D'Amour, Quadratic Jordan Systems of
Hermitian Type,  J.~Algebra 149
 (1992) 197-233.



\bibitem[FG]{fg-lattices} A. Fern\'{a}ndez L\'{o}pez, E. Garc\'{\i}a Rus, Algebraic lattices and nonasociative structures, Proc. Amer. Math. Soc. 126 (1998), 3211-3221.


\bibitem[FGGS]{goldie pares} A. Fern\'{a}ndez L\'{o}pez, E. Garc\'{\i}a Rus, M. G\'{o}mez Lozano, M. Siles Molina, Goldie theorems for associative pairs,   Communications in Algebra
   26 (9)
  (1998)  2987-3020.

\bibitem[FGJ]{towards} A. Fern\'{a}ndez L\'{o}pez, E. Garc\'{\i}a Rus,  O. Jaa, Towards a Goldie Theory for Jordan pairs,
Manuscripta Math. 95 (1998)  79-90.


\bibitem[FGM1]{fgm} A. Fern\'{a}ndez L\'{o}pez, E. Garc\'{\i}a Rus, F.
Montaner, Goldie Theory for Jordan algebras, J. Algebra 248
(2002)  397-471.

\bibitem[FGM2]{fgm_pares} A. Fern\'{a}ndez L\'{o}pez, E. Garc\'{\i}a Rus, F.
Montaner, Goldie Theory for Jordan pairs, manuscript.

 \bibitem[FT]{ft} A. Fern\'{a}ndez L\'{o}pez, M. I. Toc\'{o}n,  Strongly prime Jordan pairs with nonzero socle,  Manuscripta
 Math.
 111  (2003) 321-340.

\bibitem[GS]{gs comm alg32} M. G\'{o}mez Lozano, M. Siles Molina,  Left Quotient Associative Pairs and Morita Invariant Properties,
Comm. Algebra  32  (7) (2004) 2841-2862.



\bibitem[L1]{loos-jp}  O. Loos,  Jordan Pairs, Lecture Notes in
Mathematics, Vol. 460, Springer, New York, 1975.


\bibitem[L2]{loos-socle}  O. Loos, On the socle of a Jordan pair, Collect. Math. 40 (1989) 109-125.

\bibitem[L3]{loos-finiteness} O. Loos,  Finiteness conditions in Jordan pairs,   Math. Z. 206 (1991) 577-587.



\bibitem[L4]{loos-1995}  O. Loos, Elementary groups and stability for Jordan pairs, K. theory 9 (1995) 77-116.






\bibitem[Mc1]{mc-prime-inh}  K. McCrimmon, Strong prime
inheritance in Jordan systems, Algebras Groups and Geom. 1 (1984),
 217-234.

\bibitem[Mc2]{mc-martindale}  K. McCrimmon,  Martindale systems of symmetric quotients, Algebras Groups and Geom. 6 (1998)
191-225.

\bibitem[McZ]{mcz}   K. McCrimmon, E. Zelmanov, The Structure of Strongly
Prime Quadratic Jordan Algebras,
   Adv. Math.  69
 (1988) 133-222.
\bibitem[Me]{meyberg}  K. Meyberg,   Lectures on Jordan Algebras and
Triple Systems,  Lecture Notes, University of Virginia,
Charlottesville, 1972.
\bibitem[M1]{pi-i}   F. Montaner, Local PI Theory of Jordan Systems,
 J. Algebra  216
 (1999) 302-327.
\bibitem[M2]{pi-ii}   F. Montaner, Local PI Theory of Jordan Systems II,
  J. Algebra  241
 (2001) 473-514.


 \bibitem[M3]{densos}   F. Montaner, Maximal algebras of quotients
 of Jordan algebras, J. Algebra
 323
 (2010) 2638-2670.
  \bibitem[MP1]{esenciales}   F. Montaner, I. Paniello, Algebras of
 quotients of nonsingular Jordan algebras,   J. Algebra
 312
 (2007) 963-984.
 \bibitem[MP2]{mp-localorders} F. Montaner, I. Paniello,  Local orders in Jordan algebras, J.~Algebra 485 (2017) 45-76.
  \bibitem[MP3]{mp-lancaster} F. Montaner, I. Paniello,   PI theory for associative pairs, Linear Multilinear Agebra, to appear.

\bibitem[MT1]{mt-lc} F. Montaner, M. Tocón, Local Lesieur-Croisot theory of  Jordan algebras, J. Algebra 301 (2006) 256-273.
 \bibitem[MT2]{lc-pr} F. Montaner, M. I. Toc\'{o}n, The ideal of Lesieur-Croisot elements of a Jordan algebra. II,  Algebras, representations and applications, 199-203, Contemp. Math., 483, Amer. Math. Soc., Providence, RI, 2009.

 \bibitem[R]{rowen-pi}  L. H. Rowen,  Polynomial Identities in Ring Theory, Academic Press, New York, 1980.
 \bibitem[Z1]{z-gol1} E. Zelmanov,  Goldie's theorem for Jordan algebras,
 Siberian Math. J. 28 (1987) 44-52.
 \bibitem[Z2]{z-gol2}  E. Zelmanov,  Goldie's theorem for Jordan algebras II,
 Siberian Math. J. 29 (1988) 68-74.



\end{thebibliography}
\end{document}